\documentclass[reqno,onecolumn,oneside]{paper}
\usepackage[english]{babel}
\usepackage{enumerate,amsmath,amsthm,amsfonts,
amssymb,latexsym,mathrsfs,
}
\usepackage[all]{xy}
\usepackage[sort]{cite}

\newtheorem{theorem}{Theorem}[section]
\newtheorem{proposition}[theorem]{Proposition}
\newtheorem{lemma}[theorem]{Lemma}
\newtheorem{corollary}[theorem]{Corollary}
\theoremstyle{definition}
\newtheorem{definition}[theorem]{Definition}
\newtheorem{exm}[theorem]{Example}

\theoremstyle{remark}
\newtheorem{remark}[theorem]{Remark}

\newcommand{\bydef}{\mathrel{\mathop:}=}
\newcommand{\frab}{\operatorname{Free}_{\Gab}}
\newcommand{\End}{\operatorname{End}}

\newcommand{\op}{\operatorname{op}}

\newcommand{\id}{\operatorname{id}}

\newcommand{\rMod}{\mathcal{M}\!\!\:\mathit{od}\textrm{-}}

\newcommand{\lMod}{\textrm{-}\mathcal{M}\!\!\:\mathit{od}}
\newcommand{\Mod}{\mathcal{M}\!\!\:\mathit{od}}
\newcommand{\Id}{\operatorname{Id}}

\newcommand{\lang}{\mathcal{L}}
\newcommand{\fml}{\mathit{Fm}_\mathcal{L}}

\newcommand{\SL}{\mathcal{SL}}
\newcommand{\Q}{\mathcal{Q}_u}

\newcommand{\cat}{\mathcal}

\renewcommand{\wp}{\mathscr{P}}

\renewcommand{\phi}{\varphi}
\newcommand{\e}{\varepsilon}
\newcommand{\g}{\gamma}

\newcommand{\la}{\left\langle}
\newcommand{\ra}{\right\rangle}
\newcommand{\liff}{\Longleftrightarrow}

\newcommand{\under}{\backslash}
\newcommand{\ost}{{}_\ast/}
\newcommand{\lto}{\longrightarrow}

\newcommand{\lmapsto}{\longmapsto}
\newcommand{\ust}{\under_\ast}

\newcommand{\gr}{Grothendieck }
\newcommand{\Gab}{\cat{G}^\text{Ab}}
\renewcommand{\P}{\operatorname{P_F}}

\def\amslatex\slash{{\protect\AmS-\protect\LaTeX}}

\begin{document} 

\title{Quantales and their modules: projective objects, ideals, and congruences}

\author{Ciro Russo} 

\institution{Departamento de Matem\'atica \\ Instituto de de Matem\'atica \\ Universidade Federal da Bahia, Brazil \\ \small{\texttt{ciro.russo@ufba.br}}}

\maketitle

\today

\begin{abstract}
In this paper we characterize the projective modules over an arbitrary quantale, and then we apply such a characterization in order to define the $K_0$ group of a quantale. Then we study congruences of quantales and quantale modules by means of their ideals and of saturated elements w.r.t. a binary relation.
\end{abstract}

{\em Dedicado a Chico Miraglia, pilar da L\'ogica Matem\'atica brasileira.}

\section*{Introduction}

Since their introduction, in connection with the theory of C$^*$-algebras \cite{mulvey}, quantales proved to be extremely useful in various areas of pure and applied mathematics.

Although they are very often studied in connection with non-commutative topology (see, e.g., \cite{conmir,berni,borc1,borc2}), in the last decades quantales are appearing more and more in other areas' literature. For example, Abramsky and Vickers introduced the concepts of observational logic and process semantics, and the algebraic notion of quantale module \cite{abrvic}; later on, quantales and quantale modules were used in the study of algebraic and logical foundations of Quantum Mechanics \cite{abrvic,moore,resende1}. In \cite{russo}, the author presented, as an application, an approach to data compression algorithms by means of quantale module homomorphisms.

For what concerns Mathematical Logic, Yetter \cite{yetter} proved the connection of quantales with Girard's Linear Logic \cite{girard} and, in recent years, a quantale-theoretic approach to propositional deductive systems has been developed \cite{galtsi,russothesis,russo2}, starting from the observation that any propositional deductive system can be represented as a quantale module.

However, despite of their multiple applications, the first systematic studies on the categories of quantale modules are rather recent \cite{russothesis,russo,russo2,solo}. On the other hand, the results presented in \cite{galtsi} and \cite{russo2} clearly suggest that the algebraic categories of quantales, unital quantales, and quantale modules are worth to be further investigated.

Aim of this paper is precisely to continue the study of quantales and their modules following two main guidelines: the classical theory of ring modules, as a sort of set course, and various classical logical problems, as a source of inspiration.

More precisely, we present the following results. First, we shall improve the representation of endomorphisms of free modules, presented in \cite[Theorem 5.19]{russo}, by showing that the sup-lattice isomorphism between $Q^{X \times X}$ and $\End_Q(Q^X)$ of that representation is actually a quantale isomorphism if $Q^{X \times X}$ is equipped with a suitable product (Theorem \ref{matrixendo}). Then, using such a result, we shall characterize projective modules by means of multiplicatively idempotent elements of $Q^{X \times X}$ (Theorem \ref{proj}). We will conclude Section \ref{projsec} by observing that the construction of the $K_0$ group of a semiring, presented in \cite{dnr}, can be plainly applied to quantales, thus promisingly broadening this topic's horizons.

Sections \ref{qidsec} and \ref{qcongsec} aim at answering the question of whether it is possible to represent or, better, to recover quantale and quantale module congruences by using the congruence class of the bottom element. This question naturally leads to the definition of ideals in such structures, and Theorems \ref{icong} and \ref{qicong} describe the relationship between ideals and congruences in quantale modules and quantales respectively. Moreover, it turns out that right, left, and two-sided ideals of a quantale are in one-one correspondence, respectively, with its right-, left-, and two-sided elements. The results obtained imply immediately an interesting consequence, namely, that semisimple integral quantales are precisely the spatial frames (Corollary \ref{intsemisimple}). The section is completed by a description of quantale quotients by means of the so-called saturated elements w.r.t. a binary relation. Such results are more or less known in more restricted or different contexts, such as unital commutative quantales, quantale modules, frames, and locales, but apparently they have never been extended to the general case of arbitrary quantales.

Last, we conclude the paper with some final remarks, and an outline about current related projects.

\section{Preliminaries}

Before introducing quantales we recall that the category $\SL$ of sup-lattices has complete lattices as objects and maps preserving arbitrary joins as morphisms. The bottom element of a sup-lattice shall be denoted by $\bot$ and the top element by $\top$. We also recall that any sup-lattice morphism obviously preserve the bottom element while it may not preserve the top.

\begin{definition}\label{quantale}
A \emph{quantale} is a structure $\la Q, \bigvee, \cdot \ra$ such that
\begin{enumerate}[(Q1)]
\item $\la Q, \bigvee\ra$ is a sup-lattice,
\item $\la Q, \cdot \ra$ is a semigroup,
\item $a \cdot \bigvee B = \bigvee\limits_{b \in B} (a \cdot b)$ \ and \ $\left(\bigvee B\right) \cdot a = \bigvee\limits_{b \in B} (b \cdot a)$ \ for all $a \in Q$, $B \subseteq Q$.
\end{enumerate}
A quantale $Q$ \emph{commutative} if so is the multiplication. $Q$ is said to be \emph{unital} if there exists $1 \in Q$ such that $\la Q, \cdot, 1\ra$ is a monoid. A unital quantale is called \emph{integral} if $1 = \top$.

The morphisms in the category $\cat Q$ of (all) quantales are maps that are simultaneously sup-lattice and semigroup homomorphisms. In the category $\Q$ of unital quantales the morphisms must also preserve the unit; in order to underline this fact we will use the notation $\la Q, \bigvee, \cdot, 1\ra$ for unital quantales.
\end{definition}

\begin{definition}\label{modules}
Let $Q$ be a unital quantale. A (left) \emph{$Q$-module} $M$, or a \emph{module over $Q$}, is a sup-lattice $\la M, \bigvee\ra$ with an external binary operation, called \emph{scalar multiplication},
$$\ast: (a,v) \in Q \times M \lmapsto a \ast v \in M,$$
such that the following conditions hold:
\begin{enumerate}[(M1)]
\item $(a \cdot b) \ast v = a \ast (b \ast v)$, for all $a, b \in Q$ and $v \in M$;
\item the external product is distributive with respect to arbitrary joins in both coordinates, i.e.
\begin{enumerate}[(i)]
\item for all $a \in Q$ and $X \subseteq M$, $a \ast {}^M\bigvee X = {}^M\bigvee_{v \in X} a \ast v,$
\item for all $A \subseteq Q$ and $v \in M$, $\left({}^Q\bigvee A\right) \ast v = {}^M\bigvee_{a \in A} a \ast v,$
\end{enumerate}
\item $1 \ast v = v$.
\end{enumerate}
In the case of non-unital quantales a left module is a sup-lattice with a scalar multiplication satisfying (M1) and (M2). Right modules are defined in the usual way. Moreover, if $R$ is another quantale, a sup-lattice $M$ is a $Q$-$R$-bimodule if it is a left $Q$ module, a right $R$-module, and in addition $(a \ast_Q v) \ast_R a' = a \ast_Q (v \ast_R a')$ for all $a \in Q$, $a' \in R$, and $v \in M$.
\end{definition}

Condition (M2) is equivalent to the following one.
\begin{enumerate}
\item[(M2)]The scalar multiplication is residuated in both arguments (with respect to the lattice order in $M$), i.e. the maps
$$a \ast _-: v \in M \lmapsto a \ast v \in M \quad \textrm{and} \quad _- \ast v: a \in Q \lmapsto a \ast v \in M$$
are residuated for all $a \in Q$ and $v \in M$ respectively.
\end{enumerate}

The proof of the following proposition is straightforward from the definitions of $\ast$, $\ust$ and $\ost$ and from the properties of quantales; however, it can be found in \cite{galtsi}.

\begin{proposition}\label{basicmq}
For any unital quantale $Q$ and any $Q$-module $M$, the following hold.
\begin{enumerate}[(i)]
\item The operation $\ast$ is order-preserving in both coordinates.
\item The operations $\ust$ and $\ost$ preserve meets in the numerator; moreover, they convert joins in the denominator into meets. In particular, thew are both order-preserving in the numerator and order reversing in the denominator.
\item $(v \ost w) \ast w \leq v$.
\item $a \ast (a \ust v) \leq v$.
\item $v \leq a \ust (a \ast v)$.
\item $(a \ust v) \ost w = a \under (v \ost w)$.
\item $((v \ost w) \ast w) \ost w = v \ost w$.
\item $1 \leq v \ost v$.
\item $(v \ost v) \ast v = v$.
\end{enumerate}
\end{proposition}
Note that some of the above inequalities are in $M$ and some are in $Q$ but we used the same symbol for both. This will happen often throughout the paper since we can rely on the context telling the two relations apart.

\begin{remark}\label{notation}
Henceforth, in all the definitions and results that can be stated both for left and right modules, we will refer generically to ``modules''~--- without specifying left or right~--- and we will use the notations of left modules.
\end{remark}

Given two $Q$-modules $M$ and $N$, and a map $f: M \lto N$, $f$ is a $Q$-module homomorphism if it is a sup-lattice homomorphism that preserves the scalar multiplication. For any quantale $Q$ we shall denote by $Q$-$\Mod$ and $\Mod$-$Q$ respectively the categories of left $Q$-modules and right $Q$-modules with the corresponding homomorphisms. Moreover, if $R$ is another quantale $Q$-$\Mod$-$R$ shall denote the category whose objects are $Q$-$R$-bimodules and morphisms are maps which are simultaneously left $Q$-module morphisms and right $R$-module morphisms.

For the basic properties of quantales and their modules we refer the reader respectively to \cite{krpa,rosenthal,mulvey,mulnaw} and to \cite{krpa,russothesis,russo,solo}. In particular, we recall the following well-known facts, whose proofs can all be found in \cite[Sections 1 and 2]{krpa}.
\begin{proposition}\label{facts}
The following statements hold.
\begin{enumerate}[(a)]
\item Any quantale $Q$ can be embedded into a unital quantale $Q[e]$.
\item For any set $X$, the free quantale (respectively: unital quantale) over $X$ is the powerset of the free semigroup (resp.: free monoid) over $X$, equipped with the singleton map, with set-theoretic union as join and the product defined by $Y \cdot Z \bydef \{yz \mid y \in Y, z \in Z\}$.
\item For any unital quantale $Q$ and for any set $X$, the free (left) $Q$-module over $X$ is the function module $Q^X$, equipped with the map
$$x \in X \lmapsto e_x \in Q^X, \quad \textrm{with } e_x(y)= \left\{\begin{array}{ll} 1 & \textrm{if } y = x \\ \bot & \textrm{if } y \neq x \end{array}\right.,$$
with pointwise join and scalar multiplication.
\item For any non-unital quantale $Q$, the free (left) $Q$-module over $X$ is precisely the free (left) module over $Q[e]$.
\item The categories $\cat Q$, $\Q$, $Q$-$\Mod$ and $\Mod$-$Q$ (for any $Q$ in $\cat Q$ or $\Q$) are algebraic categories.
\end{enumerate}
\end{proposition}

For the sake of completeness, we briefly recall the construction presented in \cite[Lemma 1.1.11]{krpa}, which proves the (a) of Proposition \ref{facts}. Given a quantale $Q$, let $Q[e]$ be defined as follows:
\begin{itemize}
\item[-]$e \notin Q$;
\item[-]$Q[e] = \{a \vee \e \mid a \in Q, \e \in \{\bot, e\}\}$;
\item[-]for any family $\{a_x \vee \e_x\}_{x \in X}$ of elements of $Q[e]$;
$$\bigvee_{x \in X} (a_x \vee \e_x) \bydef \left\{\begin{array}{ll} \left({}^Q\bigvee_{x \in X} a_x\right) \vee e & \textrm{if } \exists x \ \e_x = e \\
{}^Q\bigvee_{x \in X} a_x & \textrm{otherwise}\end{array}\right.;$$
\item[-]  for all $a \vee \e, a' \vee \e' \in Q[e]$,
$$(a \vee \e) \cdot (a' \vee \e') \bydef \left\{\begin{array}{ll}
a \cdot_Q a' & \textrm{if } \e = \e' = \bot \\
a \cdot_Q a' \vee_Q a' & \textrm{if } \e = e \textrm{ and } \e' = \bot \\
a \cdot_Q a' \vee_Q a & \textrm{if } \e = \bot \textrm{ and } \e' = e \\
(a \cdot_Q a' \vee_Q a \vee_Q a') \vee e & \textrm{if } \e = \e' = e \\
\end{array}\right..$$
\end{itemize}
Then the embedding of $Q$ into $Q[e]$ is $\iota_e: a \in Q \lmapsto a \vee \bot \in Q[e]$.

It is well-known that sup-lattice homomorphisms are precisely the residuated maps. So, given two $Q$-modules $M$ and $N$, a map $f: M \lto N$ is a homomorphism if and only if it is a residuated map that preserves the scalar multiplication. In addition, let us recall that for any sup-lattice $\la M, \bigvee\ra$ the structure $M^{\op} \bydef \la M, \bigwedge\ra$ is again a sup-lattice (w.r.t. to the dual order $\geq$) and, if $M$ is a left (respectively: right) $Q$-module with scalar multiplication $\ast$, then $M^{\op}$ is a right (resp.: left) $Q$-module with scalar multiplication $\ust$. Such a correspondence, together with the properties of residuated maps and operations, immediately gives the following relevant property (see also \cite[Section 5]{abrvic}).
\begin{proposition}\label{dualiso}
For any quantale $Q$ the categories $Q\lMod$ and $\rMod Q$ are dually isomorphic.
\end{proposition}

Let $M$ and $N$ be $Q$-modules. The set $\hom_Q(M,N)$ of the $Q$-module morphisms from $M$ to $N$ is naturally equipped with a sup-lattice structure defined pointwise. Moreover, if $Q$ is commutative, $\hom_Q(M,N)$ becomes a $Q$-module with the scalar multiplication $\bullet$ defined, for all $a \in Q$ and $h \in \hom_Q(M,N)$, by setting $(a \bullet h)(x) = a \ast h(v) = h(a \ast v)$, for all $v \in M$.

\begin{remark}\label{unit}
In the present paper we are mainly intersted in unital quantales and their modules. However, some of the results that we present here either hold for all quantales or can be suitably reformulated and proved in the general setting by means of (a), (b) and (d) of Proposition \ref{facts}.

In order to keep notations as light as possible, in the rest of the paper we shall always deal with unital quantales and their modules without explicitly repeating it all the times. At the end of each section we shall discuss the extensions of the results presented to all quantales, whenever needed.
\end{remark}

\section{Projective quantale modules and $K_0$ group of a quantale}
\label{projsec}

In what follows, for any subset $S$ of $Q^X$, we shall denote by $Q \cdot S$ the submodule of $Q^X$ generated by $S$.

Let $Q \in \Q$ and $X, Y$ be non-empty sets and let us consider the free $Q$-modules $Q^X$ and $Q^Y$. We recall from \cite{russo} that, for any $k \in Q^{X \times Y}$, the \emph{$\Q$-module transform} $h_k: Q^X \lto Q^Y$ with \emph{kernel} $k$ is the defined by
\begin{equation}\label{qwtransformeq}
h_k f(y) = \bigvee_{x \in X} f(x) \cdot k(x,y), \quad \textrm{for all } y \in Y.
\end{equation}
Its \emph{inverse transform} $\lambda_k: Q^Y \lto Q^X$ is defined by
\begin{equation}\label{qwinverse transformeq}
\lambda_k g(x) = \bigwedge_{y \in Y} g(y)/k(x,y), \quad \textrm{for all } x \in X.
\end{equation}

\begin{remark}\label{righttrans}
Recalling that we are using the notations of left modules, we observe that, if we consider $Q^X$ and $Q^Y$ as right modules, the direct and inverse transforms are defined respectively by $h_k f(y) = \bigvee_{x \in X} k(x,y) \cdot f(x)$ and $\lambda_k g(x) = \bigwedge_{y \in Y} k(x,y) \under g(y)$.

Up to a suitable reformulation, all the results we will present hold for both left and right modules.
\end{remark}

\begin{theorem}\cite[Theorem 5.7]{russo}\label{wadjointpair}
Let $Q \in \Q$, $X, Y$ be two non-empty sets and $k \in Q^{X \times Y}$. The following hold:
\begin{enumerate}[(i)]
\item $(h_k,\lambda_k)$ is an adjoint pair, i.e. $h_k$ is residuated and $\lambda_k = h_{k*}$;
\item $h_k \in \hom_Q\left(Q^X,Q^Y\right)$;
\item $\lambda_k \circ h_k$ is a nucleus over $Q^X$.
\end{enumerate} 
\end{theorem}

Let us now consider the case of endomorphisms of a given free $Q$-module. The set $\End_Q(Q^X)$ has a natural structure of quantale with the pointwise join, the product of endomorphisms defined as the composition in the reverse order\footnote{Actually also the composition can be used. Here the other multiplication is needed in order to establish Theorem \ref{matrixendo}} ($h_1 h_2 \bydef h_2 \circ h_1$), and the identity of $Q^X$ as unit. Furthermore, we set $M_{X}(Q)$ to be the structure $\la Q^{X \times X}, \bigvee, \star, \id \ra$, where
\begin{itemize}
\item $\id$ is the map defined by $\id(x,y) = \left\{\begin{array}{ll} 1 & \textrm{if } x=y \\ \bot & \textrm{otherwise} \end{array}\right.$,
\item $\bigvee$ is the pointwise join,
\item the operation $\star$ is defined by $(h \star k)(x,y) = \bigvee_{z \in X} h(x,z) k(z,y)$.
\end{itemize}
It is immediate to verify that $M_{X}(Q)$ is a quantale; moreover the following result holds.

\begin{theorem}\label{matrixendo}
For any quantale $Q$ and any non-empty set $X$ the quantales $M_X(Q)$ and $\End_Q(Q^X)$ are isomorphic.
\end{theorem}
\begin{proof}
By \cite[Theorem 5.19]{russo}, the map $\eta: k \in M_X(Q) \lmapsto h_k \in \End_Q(Q^X)$ is a sup-lattice isomorphism, so we just need to prove that $\eta$ preserves the monoid structure.

First, it can be immediately observed that $\eta(\id) = \id_{Q^X}$. Now, given a map $k \in Q^{X \times X}$, for all $f \in Q^X$, $h_k(f)(y) = \bigvee_{x \in X} f(x) k(x,y)$. Hence
$$\begin{array}{l}
h_{k \star l} (f)(y) = \bigvee_{x \in X} f(x) \left(\bigvee_{z \in X} k(x,z) l(z,y)\right) \\
= \bigvee_{x \in X} \bigvee_{z \in X} (f(x) (k(x,z) l(z,y))) \\
= \bigvee_{z \in X} \left(\bigvee_{x \in X} f(x) k(x,z)\right) l(z,y) \\
= \bigvee_{z \in X} h_k(f)(z) l(z,y) \\
= (h_l \circ h_k)(f)(y),
\end{array}$$
for all $k, l \in M_X(Q)$ and $f \in Q^X$. Therefore $\eta$ is a quantale isomorphism.
\end{proof}

Thanks to Theorem \ref{matrixendo} we can now give a characterization of projective objects in the category $Q\lMod$ in terms of $Q$-valued maps.

\begin{theorem}\label{proj}
Let $M$ be $Q$-module and $X \subseteq M$ be a set of generators for $M$. $M$ is projective if and only if there exists a multiplicatively idempotent element $k$ of $M_{X}(Q)$ such that $M \cong Q \cdot \{k(x,_-)\}_{x \in X}$.
\end{theorem}
\begin{proof}
Let $M$ be projective. Since $X$ is a generating set for $M$ and $Q^X$ is free over $X$, $M$ is a retract of $Q^X$. More precisely, the identity map of $X$ can be extended to a unique $Q$-module homomorphism $\pi: Q^X \lto M$ which is obviously onto. Then the projectivity of $M$ implies the existence of a morphism $\mu: M \lto Q^X$ such that $\pi \circ \mu = \id_M$, and $\mu$ must be injective. So, if we set $k(x,_-) = \mu(x)$, for all $x \in X$, we have $M \cong Q \cdot \{k(x,_-)\}_{x \in X} = h_k[Q^X]$.

So $h_k$ is a retraction whose corresponding section is the inclusion map. Hence, for any element $f \in Q \cdot \{k(x,_-)\}_{x \in X}$, $h_k(f) = f$. In particular, for all $x \in X$, $k(x,_-) = \bigvee_{x \in X} 1 \ast k(x,_-) \in Q \cdot \{k(x,_-)\}_{x \in X}$ and therefore $h_k(k(x,_-)) = k(x,_-)$. Then we have
$$k(x,_-) = h_k(k(x,_-)) = \bigvee_{y \in X} k(x,y) \ast k(y,_-)$$
for all $x \in X$, that is, $k \star k = k$ in $M_X(Q)$.

Conversely, let $M = Q \cdot \{k(x,_-)\}_{x \in X}$ and $k \star k = k$. Any element $\alpha \in Q \cdot \{k(x,_-)\}_{x \in X}$ can be written as $\bigvee_{x \in X} a_x \ast k(x,_-)$, hence $\alpha = h_k(a_-)$. By Theorem \ref{matrixendo} we have 
$$h_k(\alpha) = h_k(h_k(a_-)) = h_{k \star k}(a_-) = h_k(a_-) = \alpha.$$
It follows that the inclusion map of $Q \cdot \{k(x,_-)\}_{x \in X}$ in $Q^X$ is a section whose corresponding retraction is $h_k$. Then $Q \cdot \{k(x,_-)\}_{x \in X}$ is a retract of a free module and therefore is projective. 
\end{proof}

It is immediate to see that finitely generated projective modules over a given quantale $\la Q, \bigvee, \cdot, 1\ra$ coincide with finitely generated projective semimodules over the idempotent semiring $\la Q, \vee, \cdot, \bot, 1\ra$ \cite{dnr}. More precisely, every finitely generated semimodule $M$ over $Q$ is complete and the external multiplication over it distributes over arbitrary joins both in $M$ and $Q$, hence $M$ is also a quantale module, with the same finite generating set, and is projective because Theorem \ref{proj}, in that case, restricts to the analogous characterization presented in \cite{dnr}. Reciprocally, every finitely generated projective $Q$-module is a finitely generated projective $Q$-semimodule, with the same generating set. Moreover, products and coproducts of finitely many quantale modules and of finitely many semimodules over semirings are constructed exactly in the same way. As a consequence, the construction of the Grothendieck group of a semiring, presented in \cite[Section 6]{dnr}, immediately extends to quantales.

So we have
\begin{definition}\label{k0}
Let $Q$ be a quantale, $\la\P(Q), \oplus, [\{\bot\}]\ra$ the Abelian monoid of isomorphism classes of projective left $Q$-modules and let $J = \frab(\P(Q))$ the free Abelian group generated by such isomorphism classes. For any projective left $Q$-module $M$, we denote by $[M]$ its isomorphism class. Let $H$ be the subgroup of $J$ generated by all the expressions of type $[M] + [N] - [M \oplus N]$, with $M$ and $N$ projective modules.

We define the \emph{\gr group} of a quantale $Q$ to be the factor group $J/H$ and denote it by $K_0Q$.
\end{definition}

\begin{theorem}\label{k0thm}
$K_0$ is a functor from $\Q$ to $\Gab$.
\end{theorem}

Theorems \ref{wadjointpair} and \ref{matrixendo} hold for non-unital quantales too, but it is important to underline that, if $Q$ is a non-unital quantale, then the free modules over $Q$ coincide with the free modules over the unital quantale $Q[e]$, exactly as in the case of ring modules. So, both the results can be stated and proved for the (non-free) modules of type $Q^X$ as well as for the free $Q$-modules $Q[e]^X$, by suitably using $Q$ or $Q[e]$.

Regarding the rest of the section, in the case of non-unital quantales, again, projective objects of $Q\lMod$ coincide with those of $Q[e]\lMod$, and therefore everything works the same way and the $K_0$ group of a non-unital quantale $Q$ coincide with the one of $Q[e]$.

\section{Ideals and congruences of quantale modules}
\label{qidsec}

In this section we shall introduce ideals of quantale modules. Once observed that the $\bot$-class of a module congruence is an ideal, we will show that, given a module $M$ and an ideal $I$ of it, it is possible to define in a canonical way a congruence whose $\bot$-class is the given ideal, and that such a congruence is not unique in general. It is, indeed, the largest congruence with that property.

\begin{definition}\label{ideal}
Let $I$ be a subset of a $Q$-module $M$. $I$ is called a \emph{$Q$-ideal} of $M$ provided
\begin{enumerate}[(i)]
\item $X \subseteq I$ implies $\bigvee X \in I$,
\item $v \in I$ and $w \leq v$ imply $w \in I$,
\item $v \in I$ implies $a \cdot v \in I$, for all $a \in Q$.
\end{enumerate}
\end{definition}

By (i), since $\varnothing \subseteq I$, $\bot = \bigvee \varnothing \in I$ for any ideal $I$; in addition, both $M$ and $\{\bot\}$ are ideals.

As usual, for any subset $S$ of a $Q$-module $M$, we will denote by $(S]$ the \emph{ideal generated by $S$}, i.e., the smallest ideal containing $S$. An ideal is called \emph{principal} if it is generated by a singleton; in this case we will write $(v]$ instead of $(\{v\}]$. Among other properties, the following result shows that all ideals of quantale modules are indeed principal.

\begin{proposition}\label{idprinc}
For any quantale $Q$ and any $Q$-module $M$, the following properties hold.
\begin{enumerate}[(i)]
\item For any $Q$-ideal $I$ of $M$, $I = [\bot, \bigvee I]$
\item For any subset $S$ of $M$, $(S] = \left[\bot, \top \ast \bigvee S\right]$. So, in particular, $(v] = [\bot, \top \ast v]$ for all $v \in M$.
\item The $Q$-ideals of $M$ are precisely the intervals $[\bot, v]$ with $v$ such that $\top \ast v = v$.
\item Every $Q$-ideal of $M$ is principal.
\item If $[\bot, v]$ is a $Q$-ideal of $M$, then $[\top,v]^{\op} \bydef \{w \in M \mid \top \geq w \geq v\}$ is a $Q$-ideal of $M^{\op}$.
\end{enumerate}
\end{proposition}
\begin{proof}
\begin{enumerate}[(i)]
\item Obvious.
\item By (i) and (iii) of Definition \ref{ideal}, $\top \ast \bigvee S \in (S]$, hence $\left[\bot, \top \ast \bigvee S\right] \subseteq (S]$ by (ii) of the same definition.

Now we must prove that $\left[\bot, \top \ast \bigvee S\right]$ is a $Q$-ideal. Since, for any $v \in M$, $\top \ast (\top \ast v) = (\top \top) \ast v = \top \ast v$, we have that $\top \ast v \leq \top \ast \left(\top \ast \bigvee S\right) = \top \ast \bigvee S$ whence $\top \ast v \in \left[\bot, \top \ast \bigvee S\right]$ for all $v \in \left[\bot, \top \ast \bigvee S\right]$. Then, for all $a \in Q$ and $v \in \left[\bot, \top \ast \bigvee S\right]$, $a \ast v \leq \top \ast v \leq \top \ast \bigvee S$ and, therefore, $a \ast v \in \left[\bot, \top \ast \bigvee S\right]$.
\item Let $v \in M$ such that $\top \ast v = v$. As in the proof of (ii), it is easy to prove that $[\bot, v]$ is a $Q$-submodule of $M$ and then, by (i), a $Q$-ideal. The converse also follows from (i).
\item It follows immediately from (iii).
\item Recall that the scalar multiplication on $M^{\op}$ is $\ust$. So, since $[\bot,v]$ is a $Q$-ideal of $M$, then $\top \ast v = v$ and therefore $v \leq \top \ust v = \bigvee\{w \in M \mid \top \ast w \leq v\}$. On the other hand, by Proposition \ref{basicmq} (ii), $\top \ust v \leq 1 \ust v = v$. Then it follows from (iii) that $[\top,v]^{\op}$ is an ideal of $M^{\op}$.
\end{enumerate}
\end{proof}

\begin{definition}\label{idealel}
Let $Q$ be a quantale and $M$ a $Q$-module. An element $v \in M$ satisfying condition (iii) of Proposition \ref{idprinc}, i.e., such that $\top \ast v = v$, will be called an \emph{ideal element}. 
\end{definition}

\begin{proposition}\label{int}
Let $Q$ be an integral quantale and let $M$ be a $Q$-module. Then every element of $M$ is ideal, i.e., $[\bot, v]$ is a $Q$-ideal for all $v \in M$.
\end{proposition}
\begin{proof}
If $Q$ is integral, $\top = 1$, hence $\top \ast v = 1 \ast v = v$ for any $v \in M$. 
\end{proof}

\begin{proposition}\label{idsub}
The set $\Id(M)$ of all the ideal elements of a $Q$-module $M$ is a sup-sublattice of $M$. Moreover, if $Q$ is commutative, $\Id(M)$ is a $Q$-submodule of $M$.
\end{proposition}
\begin{proof}
We know that $\bot \in \Id(M)$. If $\{i_x\}_{x \in X}$ is a family of ideal elements, $\top \ast \bigvee_{x \in X} i_x = \bigvee_{x \in X} (\top \ast i_x) = \bigvee_{x \in X} i_x$, so $\bigvee_{x \in X} i_x \in \Id(M)$ and $\Id(M)$ is a sup-sublattice of $M$.

Now let $Q$ be commutative. For any $a \in Q$ and $i \in \Id(M)$, $\top \ast (a \ast i) = (\top a) \ast i = (a \top) \ast i = a \ast (\top \ast i) = a \ast i$, hence $a \ast i \in \Id(M)$ and the proposition is proved.
\end{proof}

It is important to notice that a $Q$-module does not have necessarily maximal ideals, as the following example shows.
\begin{exm}\label{nomax}
Let $Q = \la [0,1], \bigvee, \ast, 0, 1\ra$, where $\ast$ is any t-norm, and consider $Q$ as a module over itself. For any $a \in [0,1]$ the interval $[0,a]$ is clearly a $Q$-ideal, then there is no maximal $Q$-ideal of $Q$.
\end{exm}

Let $Q$ be a quantale and $M$ a $Q$-module. For any $v \in M$ and any $Q$-ideal $I$ of $M$, we denote by $i^v$ the scalar $\bigvee I \ost v$. It is easy to see that
$$i^v = \bigvee \left\{a \in Q \ \mid \ a \ast v \leq \bigvee I\right\} = \bigvee \left\{a \in Q \ \mid \ a \ast v \in I\right\}.$$

\begin{lemma}\label{qxprop}
Let $I$ be a $Q$-ideal of $M$. For all $v, w \in M$, $X \subseteq M$ and $a \in Q$ the following properties hold:
\begin{enumerate}[(i)]
\item $a i^v \leq i^v$;
\item if $1 \leq a$, then $a i^v = i^v$;
\item if $v \leq w$, then $i^w \leq i^v$;
\item $i^{\bigvee X} = \bigwedge_{v \in X} i^v$;
\item $i^{a \ast v} = i^v/a$.
\end{enumerate}
\end{lemma}
\begin{proof}
\begin{enumerate}[(i)]
\item For any $a \in Q$, $(a i^v) \ast v = a \ast (i^v \ast v) \in I$, hence $a i^v \in \{b \in Q \mid b \ast v \in I\}$ and $i^v$ is the supremum of this set. Therefore $a i^v \leq i^v$.
\item By (i), $a i^v \leq i^v$. On the other hand, from $1 \leq a$ follows $i^v = 1 i^v \leq a i^v$, whence the equality.
\item By Proposition \ref{basicmq} (ii).
\item Again by Proposition \ref{basicmq} (ii).
\item For all $b \in Q$,
$$\begin{array}{l}
b \leq i^v/a = \left(\bigvee I \ost v\right)/a \ \liff \ ba \leq \bigvee I \ost v \ \liff \\
ba \ast v \leq \bigvee I \ \liff \ b \ast (a \ast v) \leq \bigvee I \ \liff \\
b \leq \bigvee I \ost (a \ast v) = i^{a \ast v}.
\end{array}$$
\end{enumerate}
\end{proof}

\begin{theorem}\label{icong}
Let $Q$ be a quantale, $M$ a $Q$-module and $\theta$ a $Q$-module congruence on $M$. Then $\bot/\theta$ is a $Q$-ideal of $M$.

Conversely, for any $Q$-ideal $I$ of $M$, the relation $\theta_I$ defined by
\begin{equation}\label{icong1}
v \theta_I w \quad \textrm{ if and only if } \quad i^v = i^w
\end{equation}
or, that is the same,
\begin{equation}\label{icong2}
v \theta_I w \quad \textrm{ if and only if } \quad i^v \ast w \vee i^w \ast v \in I,
\end{equation}
is a $Q$-module congruence and $\bot/\theta_I = I$. Moreover, if $\theta$ is a congruence on $M$ such that $\bot/\theta = I$, then $\theta \ \subseteq \ \theta_I$, that is, $\theta_I$ is the largest congruence whose class of $\bot$ is $I$.
\end{theorem}
\begin{proof}
Conditions (i) and (iii) of Definition \ref{ideal} follow trivially from the fact that $\theta$ is a module congruence; (ii) can be easily proved by observing that $v \in \bot/\theta$ and $w \leq v$ imply $w/\theta \leq v/\theta = \bot/\theta$.

For the second part of the theorem, let us show first that the two definitions (\ref{icong1}) and (\ref{icong2}) are equivalent. If $i^v = i^w$, then $i^v \cdot w = i^w \cdot v \leq \bigvee I$, hence $i^v \cdot w, i^w \cdot v \in I$ and, by the definition of ideal, $i^v \cdot w\vee i^w \cdot v \in I$. Conversely, if $v$ and $w$ verify (\ref{icong2}), $i^v \cdot w, i^w \cdot v \in I$, therefore $i^v \cdot w, i^w \cdot v \leq \bigvee I$. This means that $i^v \leq i^w = \bigvee \left\{q \in Q \ \mid \ a \cdot w \leq \bigvee I\right\}$ and $i^w \leq i^v = \bigvee \left\{q \in Q \ \mid \ a \cdot v \leq \bigvee I\right\}$, whence $i^v = i^w$.

The relation $\theta_I$ is obviously an equivalence, thus we must prove only that it preserves the operations.

If $v \theta_I w$, then $i^v = i^w$ and, by Proposition \ref{qxprop} (v), $i^{a \ast v} = i^v/a = i^w/a = i^{a \ast w}$, for all $r \in Q$. Hence $a \ast v \theta_I a \ast w$ for all $r \in Q$.

Let now $\{v_x\}_{x \in X}$ and $\{w_x\}_{x \in X}$ be two families of elements of $M$ such that $v_x \theta_I w_x$ for all $x \in X$, and let $v = \bigvee_{x \in X} v_x$ and $w = \bigvee_{x \in X} w_x$. By Proposition \ref{qxprop} (iv), $i^v = \bigwedge_{x \in X} i^{v_x} = \bigwedge_{x \in X} i^{w_x} = i^w$; then $\bigvee_{x \in X} v_x \theta_I \bigvee_{x \in X} w_x$.

Now, $v \theta_I \bot$ if and only if $i^v = i^\bot = \top$, i. e., if and only if $v \in I$, then $\bot/\theta_I = I$.

Last, we must prove that $\theta_I$ is the largest congruence such that the congruence class of $\bot$ is $I$. So, let $\theta$ be a congruence such that $\bot/\theta = I$; then, if $v \theta w$, $i^v \cdot w \theta i^v \cdot v \in I$, namely $i^v \cdot w \in I$; analogously we have that $i^w \cdot v \in I$, hence $v \theta_I w$ and the theorem is proved.
\end{proof}

The congruence $\theta_I$ is, in general, not the unique one such that the class of $\bot$ is $I$ as the following example (which refers to constructions and results presented in \cite{galtsi} and \cite{russothesis}) shows.

\begin{exm}\label{exm}
Let $\lang$ be the propositional language $\{\to\}$ and let $\vdash$ be the substitution invariant consequence relation on $\wp(\fml)$ defined by no axioms and the Modus Ponens $\frac{\alpha \ \alpha \to \beta}{\beta}$ as the unique inference rule.

Then the powerset of the set of theorems (i. e., the congruence class of the empty set) is the singleton of the empty set but, nonetheless, the inference rule makes the consequence relation non-identical. Indeed, for any $\phi, \psi \in \fml$, if $\phi \neq \psi$ then $\{\phi, \phi \to \psi\} \neq \{\phi, \phi \to \psi, \psi\}$, but their congruence classes coincide.
\end{exm}

Combining Propositions \ref{dualiso} and \ref{idprinc} with Theorem \ref{icong}, we have that an ideal $[\bot, v]$ is the class of $\bot$ for as many distinct congruences as many distinct submodules of $M^{\op}$ have $v$ as the largest element (w.r.t. the dual order). Moreover, $\left(M/\theta_I\right)^{\op}$ is the smallest and $[\top,v]^{\op}$ the largest of such submodules of $M^{\op}$.

All the results of this section hold for non-unital quantales too along with their proofs, except the few ones which cannot even be stated in that case.

\section{Quantale congruences}
\label{qcongsec}

In the present section we shall use the results of the previous one in order to describe quantale congruences, and their relationship with two-sided quantale ideals and two-sided elements. Then we will characterize semisimple integral quantales, showing that they are exactly the spatial frames and, at the end of the section, we shall describe the quotient of a quantale w.r.t. the congruence generated by a set of pairs, by means of the so-called saturated elements. Such a tool will give us the possibility to obtain information also about the smallest congruence whose $\bot$-class is a given ideal.

\begin{definition}\label{quantideal}
Let $I$ be a subset of a quantale $Q$. $I$ is called a \emph{left} (respectively: \emph{right}) \emph{ideal} of $Q$ provided
\begin{enumerate}[(i)]
\item $X \subseteq I$ implies $\bigvee X \in I$,
\item $x \in I$ and $y \leq x$ imply $y \in I$,
\item $x \in I$ implies $ax \in I$ (resp.: $xa \in I$), for all $a \in Q$.
\end{enumerate}
$I$ is a \emph{two-sided ideal} or, simply, an \emph{ideal} if it is both a left and a right ideal.
\end{definition}

Another way to see quantale ideals is to consider them as $Q$-ideals of the free $Q$-module structures of $Q$. So left ideals are basically ideals of the left $Q$-module $Q_l$ and right ideals are ideals of the right $Q$-module $Q_r$. Then the results achieved so far in this section immediately apply to left and right ideals. For what concerns two-sided ideals, the following Proposition \ref{qidprinc} can be immediately obtained by adapting the proof of Proposition \ref{idprinc}.


\begin{proposition}\label{qidprinc}
For any quantale $Q$ the following properties hold.
\begin{enumerate}[(i)]
\item For any $S \subseteq Q$, $(S] = \left[\bot, \top \cdot \left(\bigvee S\right) \cdot \top\right]$. In particular, for each $a \in Q$, $(a] = [\bot, \top \cdot a \cdot \top]$.
\item The ideals of $Q$ are precisely the intervals $[\bot, a]$ with $a$ such that $\top \cdot a \cdot \top = a$ (that is, $a$ is a two-sided element of $Q$).
\item Every ideal of $Q$ is principal.
\end{enumerate}
\end{proposition}

According to Proposition \ref{qidprinc}, the principal generators of ideals of a quantale $Q$ are precisely the two-sided elements\footnote{In fact, they are strictly two-sided since in unital quantales the two notions coincide.} of a quantale $Q$; we shall denote by $\Id(Q)$ the set of all such elements and, as in the case of modules, we shall also call them \emph{ideal elements}. We observe explicitly that the unique two-sided element $a$ such that $1 \leq a$ is $\top$. Indeed from $1 \leq a$ follows $\top = \top 1 \leq \top a = a$, and therefore $a = \top$.

The following result readily follows from Theorem \ref{icong}.

\begin{theorem}\label{qicong}
For any quantale $Q$ and any ideal element $i$ of $Q$, the relation $\theta_i$ defined by
\begin{equation}\label{qicong1}
a \theta_i b \quad \textrm{ if and only if } \quad i/a = i/b \textrm{ and } a\under i = b \under i
\end{equation}
or, that is the same,
\begin{equation}\label{qicong2}
a \theta_i b \quad \textrm{ if and only if } \quad (i/a)b \vee (i/b)a \vee b(a\under i) \vee a(b \under i) \leq i,
\end{equation}
is a quantale congruence and $\bot/\theta_i = [\bot, i]$. Moreover, if $\theta$ is a congruence on $Q$ such that $\bot/\theta = [\bot, i]$, then $\theta \ \subseteq \ \theta_i$, that is, $\theta_i$ is the largest congruence whose class of $\bot$ is the downset of $i$.
\end{theorem}

Any element of an integral quantale is obviously two-sided, hence we have the following immediate corollaries.
\begin{corollary}\label{intsimple}
An integral quantale is simple if and only if it is either trivial (i.e. the one-element quantale) or the two-element chain $\{\bot, \top\}$.
\end{corollary}
\begin{corollary}\label{intsemisimple}
An integral quantale is semisimple if and only if it is a spatial frame.
\end{corollary}
\begin{proof}
It suffices to observe that Corollary \ref{intsimple} implies that a semisimple integral quantale is isomorphic to a subframe of the frame reduct of a Boolean algebra of type $\{\bot,\top\}^X$ for some set $X$, that is, a spatial frame.
\end{proof}

\begin{theorem}\label{qidsub}
The set $\Id(Q)$ of all the ideal elements of a quantale $Q$ is a non-unital subquantale of $Q$ (i.e., it is closed under $\bigvee$ and $\cdot$). It is an integral quantale. Moreover the following are equivalent:
\begin{enumerate}[(a)]
\item $\Id(Q)$ is a unital subquantale of $Q$;
\item $Q$ is integral;
\item $\Id(Q) = Q$.
\end{enumerate}
\end{theorem}
\begin{proof}
For any $a \in \Id(Q)$ obviously $a = \top a = a \top$. So, for all $a, b \in \Id(Q)$, $\top a b \top = a b \in \Id(Q)$.

Now, if $\Id(Q)$ is a subquantale of $Q$ then $1 \in \Id(Q)$ and therefore $1 = \top 1 \top = \top$, i. e., $Q$ is integral. On the other hand, if $Q$ is integral $1 = \top$, hence $a = \top a \top$ for all $a \in Q$ and $\Id(Q) = Q$. Last, it is obvious that (c) implies (a).
\end{proof}

We conclude this section with a useful description of quantale quotients w.r.t. congruence generated by a given binary relation, by means of the so-called saturated elements. The technique presented here already appeared in \cite{bapu}, in the contest of unital commutative quantales, and in \cite{paseka}, for quantale modules, and is quite common in the literature of frames and locales. However, to the best of our knowledge, a complete presentation of the topic for quantales in general has never appeared.

\begin{definition}\label{satel}
Let $Q$ be a (not necessarily unital) quantale, and $R$ be a binary relation on $Q$, i.e., a subset of $Q^2$. An element $s$ of $Q$ is called \emph{$R$-saturated} if, for all $(a,b) \in R$ and $c,d \in Q$, the following conditions hold:
\begin{enumerate}[(i)]
\item $cad \leq s$ iff $cbd \leq s$;
\item $ac \leq s$ iff $bc \leq s$;
\item $ca \leq s$ iff $cb \leq s$;
\item $a \leq s$ iff $b \leq s$.
\end{enumerate}
We shall denote by $Q_R$ the set of $R$-saturated elements of $Q$.
\end{definition}

\begin{remark}\label{satuni}
If $Q$ is unital, conditions (ii--iv) of Definition \ref{satel} are redundant, since they are all immediate consequences of (i). In the rest of this section, in order to keep the presentation reasonably concise, we shall only deal with unital quantales, and therefore only condition (i) will be used. Anyway, all of the results hold for non-unital quantales too, up to a trivial (but somewhat lenghty) extension of the proofs.
\end{remark}

\begin{proposition}\label{satinfres}
For any quantale $Q$ and for all binary relation $R$ on it, $Q_R$ is closed w.r.t. arbitrary meets. Moreover, for all $s \in Q_R$ and for all $q \in Q$, both $s/q$ and $q \under s$ belong to $Q_R$.
\end{proposition}
\begin{proof}
Let $S \subseteq Q_R$, $(a,b) \in R$, and $c,d \in Q$. We have
$$cad \leq \bigwedge S \iff \forall s \in S(cad \leq s) \iff \forall s \in S(cbd \leq s) \iff cbd \leq \bigwedge S.$$
Similarly, if $Q$ is non-unital, we get $a \leq \bigwedge S$ iff $b \leq \bigwedge S$, whence $\bigwedge S \in Q_R$ for all $S \subseteq Q_R$.

Now let $c, d, q \in Q$, $(a,b) \in R$, and $s \in Q_R$. Then
$$\begin{array}{l}
cad \leq s/q \iff cadq \leq s \iff cbdq \leq s \iff cbd \leq s/q, \\
cad \leq q \under s \iff qcad \leq s \iff qcbd \leq s \iff cbd \leq q \under s, \\
a \leq s/q \iff aq \leq s \iff cbdq \leq s \iff cbd \leq s/q, \\
cad \leq q \under s \iff qcad \leq s \iff qcbd \leq s \iff cbd \leq q \under s, \\
\end{array}
$$
So the assertion is proved.
\end{proof}

\begin{lemma}\label{RR'}
If $R \subseteq R' \subseteq Q^2$, then $Q_{R'} \subseteq Q_R$.
\end{lemma}
\begin{proof}
Trivially, if $s \in Q$ is $R'$-saturated, then conditions (i--iv) of Definition \ref{satel} hold for all $(a,b) \in R'$ and, therefore, for all $(a,b) \in R$. Hence $s \in Q_{R'}$ implies $s \in Q_R$.
\end{proof}

\begin{lemma}\label{satnuc}
Let $Q$ and $Q'$ be quantales, and $f: Q \to Q'$ a homomorphism with residuum $f_\ast: Q' \to Q$ and associated nucleus $\g = f_\ast \circ f$. Then $Q_\g$ coincide with the set of $(\ker f)$-saturated elements of $Q$. 
\end{lemma}
\begin{proof}
First, recall that the properties of residuated maps guarantee that, for all $q \in Q$, $\g(q) = \max\{x \in Q \mid f(x) \leq f(q)\}$. By definition, an element $s$ of $Q$ is $(\ker f)$-saturated if, for all $a,b,c,d \in Q$ such that $f(a) = f(b)$, $cad \leq s$ iff $cbd \leq s$. Now, if $f(a) = f(b)$ and $cad \leq \g(q)$ for some $c,d,q \in Q$, then $f(cad) \leq f(\g(q)) = f(f_\ast(f(q))) = f(q)$ and therefore $f(cbd) = f(c)f(b)f(d) = f(c)f(a)f(d) = f(cad) \leq f(q)$, from which we deduce $cbd \leq \g(q)$. The inverse implication is completely analogous, hence $\g(q)$ is $(\ker f)$-saturated, for all $q \in Q$, namely, $Q_\g \subseteq Q_{\ker f}$.

Conversely, let $s \in Q_{\ker f}$. Since $f(s) = f(f_\ast(f(s))) = f(\g(s))$, $(s,\g(s)) \in \ker f$ and therefore we have $s \leq s$ iff $\g(s) \leq s$, form which we get immediately $\g(s) \leq s$. On the other hand, $a \leq \g(a)$ for all $a \in Q$, hence $s = \g(s) \in Q_\g$, and the assertion is proved.
\end{proof}

\begin{theorem}\label{satquo}
Let $Q$ be a quantale, $R \subseteq Q^2$, and
$$\rho_R: a \in Q \mapsto \bigwedge\{s \in Q_R \mid a \leq s\} \in Q.$$
Then $\rho_R$ is a quantic nucleus whose image is $Q_R$. Moreover, $Q_R$, with the structure induced by $\rho_R$, is isomorphic to the quotient of $Q$ w.r.t. the congruence generated by $R$.
\end{theorem}
\begin{proof}
By Proposition \ref{satinfres}, $\rho_R[Q] \subseteq Q_R$. On the other hand, obviously, $\rho_R(s) = s$ for all $s \in Q_R$, and therefore $\rho_R[Q] = Q_R$. It self-evident also that $\rho_R$ is monotone, extensive, and idempotent w.r.t. composition, i.e. it is a closure operator. So, in order to prove that $\rho_R$ is a quantic nucleus, we only need to show that $\rho_R(a)\rho_R(b) \leq \rho_R(ab)$ for all $a, b \in Q$.

Let $s \in Q_R$ and let $a, b \in Q$. We have $\rho_R(ab) \leq s$ iff $ab \leq s$ iff $a \leq s/b$ iff $\rho_R(a) \leq s/b$ iff $\rho_R(a)b \leq s$ iff $b \leq \rho_R(a) \under s$ iff $\rho_R(b) \leq \rho_R(a) \under s$ iff $\rho_R(a)\rho_R(b) \leq s$. Then $\rho_R(a)\rho_R(b) \leq \rho_R(ab)$ for all $a, b \in Q$.

Now, once proved that $\rho_R$ is a quantic nucleus, we can consider $Q_R$ with its quantale structure induced by $\rho_R$, and we have that the mapping $a \in Q \mapsto \rho_R(a) \in Q_R$ is an onto homomorphism (that we will still denote by $\rho_R$). By Lemma \ref{satnuc}, we get $Q/\ker\rho_R \cong Q_{\rho_R} = Q_R = Q_{\ker\rho_R}$. Since $R \subseteq \ker\rho_R$, if $\theta$ is the congruence generated by $R$, then $\theta \subseteq \ker\rho_R$. Denote by $p_\theta$ the natural projection of $Q$ over $Q/\theta$ and by $\g$ the quantic nucleus on $Q$ induced by $p_\theta$. Then, by Lemma \ref{satnuc}, $Q_\theta = Q_\g \cong Q/\theta$. Hence, by Lemma \ref{RR'} and the first part of this proof, we obtain $Q/\ker\rho_R \cong Q_{\rho_R} = Q_{\ker\rho_R} \subseteq Q_\theta \subseteq Q_R = Q_{\ker\rho_R}$. The assertion follows.  
\end{proof}

Next result is an easy consequence of Theorem \ref{satquo}. Dually to Theorem \ref{qicong}, it points out the smallest congruence associated to a given ideal.

\begin{corollary}\label{satid}
Let $Q$ be a quantale, $i \in \Id(Q)$, and $R = \{(\bot, i)\}$. Then $Q_R$ is isomorphic to the quotient of $Q$ w.r.t. the smallest congruence whose $\bot$-class is $[\bot,i]$. In other words, such a congruence is precisely $\ker\rho_R$. 
\end{corollary}

Now observe that, according to Definition \ref{satel}, given $i \in \Id(Q)$ and $R = \{(\bot,i)\}$, an element $s$ of $Q$ is $R$-saturated if and only if, for all $c,d \in Q$, $\bot = c \bot d \leq s$ iff $cid \leq s$, from which we have that $s \in Q_R$ iff $cid \leq s$ for all $c,d \in Q$. On the other hand, since $i \in \Id(Q)$, we have that $cid \leq \top i \top = i$ for all $c,d \in Q$. This means that $i \in Q_R$, and $s \in Q_R$ iff $i \leq s$. So we have
\begin{corollary}\label{itop}
For all $i \in \Id(Q)$, the set of $\{(\bot,i)\}$-saturated elements of $Q$ is precisely $[i,\top]$.
\end{corollary}

\section{Conclusion}
\label{nonunit}

The results of the present work, and especially the ones of Section \ref{qcongsec}, represent, in our opinion, a step in the direction of a description of the lattices of congruences of quantales. Such a description, on its turn, could be extremely useful for a representation theorem for quantales which could be more handy than the few known ones (see, e. g., \cite{brgu} and \cite{val}).

The interest for the relationship between congruences and ideals was actually suggested also by the results of \cite{galtsi} and \cite{russo2}, and by Example \ref{exm}. Indeed, in the representation of deductive systems by means of quantale modules, consequence relations correspond to module congruences, and the set of theorems of a given deductive system to the congruence class of the bottom element, i. e., to an ideal.

Therefore one can ask which consequence relations definable on a given propositional language do correspond to the congruences defined in Theorem \ref{icong}. During the workshop in honour of Francisco Miraglia, the present author conjectured some relation between such congruences and the consequence relations which satisfy the Deduction and Detachment Theorem. As a remark, Miraglia himself suggested to try to characterize also those congruences which correspond to consequence relations satisfying Craig Interpolation. Both these questions do not have a definitive answer yet, and are currently object of study by the author.

\end{document}